\documentclass[lineno]{biometrika}

\usepackage{amsmath,amsfonts,amssymb}
\usepackage{times}
\usepackage{bm}
\usepackage{natbib}

\usepackage[plain,noend]{algorithm2e}

\makeatletter
\renewcommand{\algocf@captiontext}[2]{#1\algocf@typo. \AlCapFnt{}#2} 
\def\@algocf@capt@plain{top}
\renewcommand{\algocf@makecaption}[2]{%
  \addtolength{\hsize}{\algomargin}%
  \sbox\@tempboxa{\algocf@captiontext{#1}{#2}}%
  \ifdim\wd\@tempboxa >\hsize
    \hskip .5\algomargin%
    \parbox[t]{\hsize}{\algocf@captiontext{#1}{#2}}
  \else%
    \global\@minipagefalse%
    \hbox to\hsize{\box\@tempboxa}
  \fi%
  \addtolength{\hsize}{-\algomargin}%
}
\makeatother

\def\tr{\text{tr}}

\title{A simple regression equivalence of Pillai's trace statistic}

\author{Xia Shen, Zheng Ning, Yudi Pawitan}
\affil{Department of Medical Epidemiology and Biostatistics, Karolinska Institutet\\ SE-17 177 Stockholm, Sweden}

\begin{document}

\maketitle

Consider Pillai's trace $V^{(1)}$ statistic (denoted as $V$ hereafter) \citep{pillai55} for a multivariate analysis of variance (MANOVA) problem with an independent variable ${x}_{n\times1}$ and multiple dependent variables ${Y}_{n\times k}$, where ${Y}$ is column-full-ranked. Let us reverse the problem as a linear multiple regression
\begin{equation}
	{x} = a{1} + {Y}{b} + {e}.
	\label{eq1}
\end{equation}
Let $ \hat{{b}}=({Y}^{T}{Y})^{-1}{Y}^{T}{x} $ be the least-squares estimate of $ {b} $. Define a score $ {s}_{n\times 1} = {Y}\hat{{b}} $ as a linear combination of the variables in $ {Y} $, and fit a simple regression model
\begin{equation}
	{s} = \mu{1} + \beta {x}+ {\epsilon}.
	\label{eq2}
\end{equation}

\begin{theorem}
\label{theorem1}
Let $ V $ be Pillai's trace statistic for MANOVA of ${Y}$ on ${x}$, and $ \hat{\beta} $ be the least-squares estimate of $ \beta $, then $V = \hat{\beta}$.
\end{theorem}

\begin{proof}[of Theorem~\ref{theorem1}]

Defining ${A} = ({1},{Y})$, we have

\begin{equation}
	{A}^{T}{A} = \left(
	\begin{matrix}
	n & {1}^{T}{Y}\\
	{Y}^{T}{1} & {Y}^{T}{Y}
	\end{matrix}\right).\nonumber
\end{equation}

Then the inverse of ${A}^{T}{A}$ is

\begin{equation*}
	({A}^{T}{A})^{-1} = 
	\left(\begin{matrix}
	\text{F}_{11} & -\text{F}_{11}{1}^{T}{Y}\left({Y}^{T}{Y}\right)^{-1}\\
	-\frac{1}{n}{F}_{22}{Y}^{T}{1} & {F}_{22}
	\end{matrix}\right),
\end{equation*}

where $\text{F}_{11} = \{n - {1}^{T}{Y}\left({Y}^{T}{Y}\right)^{-1}{Y}^{T}{1}\}^{-1}$ and ${F}_{22} = \left({Y}^{T}{Y} - {Y}^{T}{1}{1}^{T}{Y} / n\right)^{-1}$.

Thus 

\begin{equation*}
	\left(\begin{matrix}
	\hat{a} \\ \hat{{b}}
	\end{matrix}\right)
	 = ({A}^{T}{A})^{-1}{A}^{T}{x} = 
	\left(\begin{matrix}
	\text{F}_{11}{1}^{T}{x} - \text{F}_{11}{1}^{T}{Y}\left({Y}^{T}{Y}\right)^{-1}{Y}^{T}{x}\\
	-\frac{1}{n}{F}_{22}{Y}^{T}{1}{1}^{T}{x} + {F}_{22}{Y}^{T}{x}
	\end{matrix}\right),
\end{equation*}

and 
\begin{align*}
	{s} &= {Y}\hat{{b}} = -\frac{1}{n}{Y}{F}_{22}{Y}^{T}{1}{1}^{T}{x} + {Y}{F}_{22}{Y}^{T}{x}.
\end{align*}

Let ${B} = \left(\begin{matrix}{1},{x}\end{matrix}\right)$, so that
\begin{equation}
	{B}^{T}{B} = \left(
	\begin{matrix}
	n & {1}^{T}{x}\\
	{x}^{T}{1} & {x}^{T}{x}
	\end{matrix}\right),\nonumber
\end{equation}

\begin{equation}\label{bbt}
	\det(B^{T}B) = nx^{T}x - x^{T}11^{T}x,
\end{equation}

\begin{equation*}
	\left({B}^{T}{B}\right)^{-1}{B}^{T} = 
	\frac{1}{\det\left({B}^{T}{B}\right)}
	\left(\begin{matrix}
		{x}^{T}{x}{1}^{T} - {x}^{T}{1}{x}^{T}\\
		-{1}^{T}{x}{1}^{T} + n{x}^{T}
	\end{matrix}\right).
\end{equation*}

Therefore
\begin{align*}
	\hat{\beta} &= \left(\begin{matrix}
	0, 1\\
	\end{matrix}\right)\left({B}^{T}{B}\right)^{-1}{B}^{T}{s}\\
	&= \frac{1}{\det\left({B}^{T}{B}\right)}\left(-{1}^{T}{x}{1}^{T} + n{x}^{T}\right)\left(-\frac{1}{n}{Y}{F}_{22}{Y}^{T}{1}{1}^{T}{x} + {Y}{F}_{22}{Y}^{T}{x}\right).
\end{align*}

By definition, Pillai's trace $V = \text{tr}\{\left({T} - {E}\right){T}^{-1}\}$, where
\begin{align}
	{E} &= {Y}^{T}\left\{{I} - {B}\left({B}^{T}{B}\right)^{-1}{B}^{T}\right\}{Y}, \nonumber\\
	\label{T}
	{T} &= {Y}^{T}\left({I} - \frac{1}{n}{1}{1}^{T}\right){Y}.
\end{align}

Hence
\begin{equation}\label{T-E}
	{T} - {E} = {Y}^{T}\left\{{B}\left({B}^{T}{B}\right)^{-1}{B}^{T} - \frac{1}{n}{1}{1}^{T}\right\}{Y}. 
\end{equation}

Notice in \eqref{T}, ${T}^{-1} = {F}_{22}$, so combining with \eqref{bbt} and \eqref{T-E}, we get
\begin{eqnarray*}
	V &=& \text{tr}\left\{\left({T} - {E}\right){T}^{-1}\right\}\\
	&=& \tr\left\{{Y}^{T}\left({B}\left({B}^{T}{B}\right)^{-1}{B}^{T} - \frac{1}{n}{1}{1}^{T}\right){Y}{F}_{22} \right\}\\
	&=& \tr\left\{\left({B}\left({B}^{T}{B}\right)^{-1}{B}^{T} - \frac{1}{n}{1}{1}^{T}\right){Y}{F}_{22}{Y}^{T} \right\}\\
	&=&
	\tr\left\{\left(\frac{1}{\det\left({B}^{T}{B}\right)}\left({1}{x}^{T}{x}{1}^{T} - {1}{1}^{T}{x}{x}^{T}\right) - \frac{1}{n}{1}{1}^{T}\right){Y}{F}_{22}{Y}^{T} \right\} \\
	&\quad+& \tr\left\{\frac{1}{\det\left({B}^{T}{B}\right)}\left(-{x}{1}^{T}{x}{1}^{T} + n{x}{x}^{T}\right){Y}{F}_{22}{Y}^{T}\right\}\\
	&=& 
	\tr\left\{\frac{1}{\det\left({B}^{T}{B}\right)}\left({1}{x}^{T}{x}{1}^{T} - \frac{1}{n}{1}\det\left({B}^{T}{B}\right){1}^{T} - {1}{1}^{T}{x}{x}^{T}\right){Y}{F}_{22}{Y}^{T} \right\} \\
	&\quad+& \tr\left\{\frac{1}{\det\left({B}^{T}{B}\right)}\left(-{x}{1}^{T}{x}{1}^{T} + n{x}{x}^{T}\right){Y}{F}_{22}{Y}^{T}\right\}\\
	&=&
	\frac{1}{\det\left({B}^{T}{B}\right)}\tr\left\{\left(\frac{1}{n}{1}{1}^{T}{x}{1}^{T}{x}{1}^{T} - {1}{1}^{T}{x}{x}^{T}\right){Y}{F}_{22}{Y}^{T} \right\}\\
	&\quad+& \frac{1}{\det\left({B}^{T}{B}\right)}\tr\left\{\left(-{1}^{T}{x}{1}^{T} + n{x}^{T}\right){Y}{F}_{22}{Y}^{T}{x}\right\}\\
	&=& \frac{1}{\det\left({B}^{T}{B}\right)}\tr\left\{\left({1}^{T}{x}{1}^{T} - n{x}^{T}\right)\frac{1}{n}{Y}{F}_{22}{Y}^{T}{1}{1}^{T}{x} \right\}\\
	&\quad+& \frac{1}{\det\left({B}^{T}{B}\right)}\tr\left\{\left(-{1}^{T}{x}{1}^{T} + n{x}^{T}\right){Y}{F}_{22}{Y}^{T}{x}\right\}\\
	&=& \frac{1}{\det\left({B}^{T}{B}\right)}\left(-{1}^{T}{x}{1}^{T} + n{x}^{T}\right)\left(-\frac{1}{n}{Y}{F}_{22}{Y}^{T}{1}{1}^{T}{x} + {Y}{F}_{22}{Y}^{T}{x}\right)\\
	&=& \hat{\beta}
\end{eqnarray*}
\end{proof}

Given the numerical equivalence of $V$ and $\hat{\beta}$, $\hat{\beta}$ has the same distribution as $V$. Note that the correlation coefficient between the fitted values and original response of regression (\ref{eq1}) is the same as that of regression (\ref{eq2}). Assuming each column of $Y$ follows a Gaussian distribution, let $R^{2}$ be the coefficient of determination of regressions (\ref{eq1}) and (\ref{eq2}), then in regression (\ref{eq2}), we have
\begin{equation*}
R^{2} = \hat{\beta}\frac{Cov({x},{s})}{V({s})}
\end{equation*}
Since ${x} = \hat{a}{1} + {s} + \hat{{e}}$ and $Cov({s}, \hat{a}{1} + \hat{{e}})=0$, together with Theorem (\ref{theorem1}), 
\begin{equation*}
\hat{\beta} = V = R^{2}
\end{equation*}
The F-statistic of the multiple regression (\ref{eq1}) can be expressed as a function of $R^{2}$, i.e.
\begin{equation*}
F = \frac{R^{2}/k}{(1 - R^{2})/(n - k - 1)}
\end{equation*}
which is the same formula for the F-statistic of Pillai's trace $V$ \citep{pillai60}. After rearranging,
\begin{equation*}
\hat{\beta} = V = R^{2} = \frac{kF}{(n - k - 1) + kF}
\end{equation*}
As $F\sim F(k,n-k-1)$, we have
\begin{equation*}
\hat{\beta} = V = R^{2} \sim Beta\left(\frac{k}{2}, \frac{n - k - 1}{2}\right)
\end{equation*}
which is the exact distribution of $\hat{\beta}$. In practice, the standard error of $\hat{\beta}$ can be obtained by Gaussian approximation of the Beta distribution, which simplifies the significance test of $\hat{\beta}$ as a Wald test. 

$\hat{\beta}$ is therefore a simple linear regression effect for a multivariate analysis. This single effect $\hat{\beta}$ is particularly useful in multivariate biological studies, so that the biomarker effect can be interpreted and replicated with meaning, i.e. the effect on the score $s$.

\section*{Author contributions}
X.S. initiated and coordinated the study. Z.N. and X.S. proved the theorems. X.S., Z.N. and Y.P. wrote the manuscript and approved the final version.

\bibliographystyle{plainnat}
\bibliography{pillai}

\end{document}